\documentclass[reqno]{amsart}
\usepackage[latin1]{inputenc}
\usepackage{amsmath,enumerate,amsthm,graphicx,color,verbatim,hyperref}
\usepackage{amsaddr}
\usepackage{amsfonts,}
\usepackage{latexsym}
\usepackage{amssymb}
\usepackage{bbm}

\allowdisplaybreaks

\newtheorem{thm}{Theorem}
\newtheorem{corollary}{Corollary}

\newcommand{\Z}{{\mathbbm Z}}
\newcommand{\R}{{\mathbbm R}}
\newcommand{\C}{{\mathbbm C}}
\newcommand{\N}{{\mathbbm N}}

\newcommand{\D}{{\mathbbm D}}

\newcommand{\X}{{\mathbbm S}}

\newcommand{\ECMV}{{\mathcal E}}
\newcommand{\E}{\mathbbm E}

\newcommand{\dd}{{\mathrm{d}}}

\newcommand{\ac}{{\mathrm{ac}}}

\newcommand{\CMV}{{\mathcal C}}

\newcommand{\set}[1]{\left\{#1\right\}}
\newcommand{\eqdef}{\overset{\mathrm{def}}=}

\definecolor{purple}{rgb}{.5,0,1}
\definecolor{orange}{rgb}{1,.5,0}

\begin{document}
\title[Purely AC spectrum using CMV density of states]{A condition for purely absolutely continuous spectrum for CMV operators using the\\ density of states}

\author{Jake Fillman}
\address{Virginia Tech,\\
Mathematics (MC0123), \\
225 Stanger Street, \\
Blacksburg, VA 24061, \\
USA}
\email{fillman@vt.edu}

\author{Darren C. Ong}
\address{Xiamen University Malaysia,\\
Jalan Sunsuria, Bandar Sunsuria,\\
43900 Sepang, Selangor Darul Ehsan,\\
Malaysia}
\email{darrenong@xmu.edu.my}

\maketitle

\begin{abstract}
We prove an averaging formula for the derivative of the absolutely continuous part of the density of states measure for an ergodic family of CMV matrices. As a consequence, we show that the spectral type of such a family is almost surely purely absolutely continuous if and only if the density of states is absolutely continuous and the Lyapunov exponent vanishes almost everywhere with respect to the same. Both of these results are CMV operator analogues of theorems obtained by Kotani for Schr\"odinger operators.
\end{abstract}

\section{Introduction}

\subsection{Background and Motivation}

CMV matrices are unitary operators on $\ell^2(\Z)$ or $\ell^2(\N)$ that have attracted substantial interest in recent years. These operators were introduced in \cite{CMV} as a bridge between spectral theory and the theory of orthogonal polynomials on the unit circle, an idea extensively explored in the monographs \cite{S1,S2}. In \cite{CGMV}, CMV operators were also proposed as a model for understanding one-dimensional quantum walks (quantum mechanical analogues of classical random walks). This connection has drawn a lot of attention recently from both mathematicians and physicists. See for instance \cite{ABJ2015, BGVW, CGMV2012RMP, CGVWW2016, DEFHV, DFO, DFV, DMY2, EKOS, Fillman-Ong, FOZ, GVWW, Joye-Merkli, Konno2010, Konno-Segawa2011, Konno-Segawa2014} in addition to the foregoing references. Moreover, as observed in \cite{DMY2}, one can relate the classical ferromagnetic Ising model in one dimension to a suitable CMV operator, a connection that allows one to rigorously prove characteristics such as the absence of phase transitions \cite{DFLY1}.

From the perspective of quantum walks, understanding the spectral type of CMV operators is particularly important, as the spectral type influences the scattering behavior (or lack thereof) for the associated quantum walk. A more precise relationship is furnished by a discrete-time variant of the RAGE Theorem (a name coined by Barry Simon to reflect the contributions of \cite{AmrGeo1973/74,Enss1978CMP,Ruelle}). Such a  discrete-time formulation of the RAGE Theorem suitable for quantum walks (with proofs) may be found in the Appendix of \cite{Fillman-Ong}. Roughly, point spectrum is associated with localization, that is, the walker remains close to its starting position; absolutely continuous spectrum is associated with scattering, that is, the walker flees to infinity; and, most exotically, singular continuous spectrum is associated with recurrent scattering, for which the walker flees to infinity in a time-averaged sense, but may potentially recur to its initial position along a subsequence of time scales. Thus, determining when a CMV operator possesses purely absolutely continuous spectrum is helpful in determining when scattering occurs in a quantum walk.

To this end, in this paper we adapt to the CMV operator some ideas originally used to solve related problems in the theory of Schr\"odinger operators in $L^2(\R)$, that is, second-order differential operators of the form
\[
H\psi
=
-\psi'' + V\psi.
\]

For such operators, \emph{Kotani theory} refers to a far-reaching family of results of Shinichi Kotani regarding the absolutely continuous spectrum of $H$ whenever $V$ is a metrically transitive potential (in more modern parlance:\ whenever $V$ is an ergodic, dynamically defined potential) \cite{Kotani1984,Kotani1985CMP,Kotani1987,Kotani1997}. Barry Simon generalized Kotani's work to the setting of Jacobi \cite{Simon1983CMP} and CMV \cite{S2} operators, with one exception. In the Notes for \cite[Section~10.11]{S2}, Simon points out that Kotani proves a stronger version of \cite[Theorem~10.11.2]{S2}, namely, that an ergodic family of Schr\"odinger operators exhibits purely absolutely continuous spectrum (almost surely with respect to the underlying ergodic measure) if and only if the density of states (DOS) measure is absolutely continuous and the Lyapunov exponent vanishes a.e.\ with respect to the DOS measure; see \cite[Corollary~4.8.2]{Kotani1997}. Simon notes that such a result also ought to hold true for ergodic families of CMV operators. The aim of this work is to prove exactly this theorem.

The key ingredient in Kotani's proof of this statement for Schr\"odinger operators is an averaging formula, found in \cite[Theorem~4.8]{Kotani1997}. Broadly speaking, Kotani's formula identifies the absolutely continuous part of the DOS with the average of the absolutely continuous parts of the associated spectral measures. In the present paper, we prove a suitable version of Kotani's averaging formula for CMV matrices, and then deduce his characterization of purely absolutely continuous spectrum as a consequence.

\subsection{Results}
An \emph{extended CMV matrix} is a pentadiagonal unitary operator on $\ell^2(\Z)$ with a repeating $2 \times 4$ block structure of the form
\begin{equation} \label{eq:stdcmvdef}
\ECMV
=
\ECMV_\alpha
=
\begin{bmatrix}
\ddots & \ddots & \ddots & \ddots &&& \\
& \overline{\alpha_2}\rho_1 
& -\overline{\alpha_2}\alpha_1 
& \overline{\alpha_3} \rho_2 
& \rho_3\rho_2 &\\
& \rho_2\rho_1 
& -\rho_2\alpha_1 
& -\overline{\alpha_3}\alpha_2 
& -\rho_3\alpha_2 & \\
&& \ddots & \ddots & \ddots & \ddots
\end{bmatrix},
\end{equation}
where $\alpha_n \in \D \eqdef \{ z \in \C : |z| < 1\}$ and $\rho_n = \sqrt{1-|\alpha_n|^2}$ for all $n \in \Z$. Setting $\alpha_{-1} = -1$, the operator decouples into two half-line operators. The operator on the right half-line takes the form
\begin{equation} \label{eq:halflinecmvdef}
\CMV
=
\begin{bmatrix}
 \overline{\alpha_0} & \overline{\alpha_1}\rho_0 & \rho_1\rho_0 && \\
\rho_0 & -\overline{\alpha_1}\alpha_0 & -\rho_1 \alpha_0 && \\
& \overline{\alpha_2}\rho_1 & -\overline{\alpha_2}\alpha_1 & \overline{\alpha_3} \rho_2 & \rho_3\rho_2 \\
& \rho_2\rho_1 & -\rho_2\alpha_1 & -\overline{\alpha_3}\alpha_2 & -\rho_3\alpha_2  \\
&& \ddots & \ddots &  \ddots & \ddots
\end{bmatrix},
\end{equation}
and is known as a \emph{standard} or \emph{half-line CMV matrix.}

Let us briefly recall how the density of states is defined. Given a CMV operator $\ECMV$ and $n \in \Z_+$, we denote its restriction to $[-n,n]$ with Dirichlet boundary conditions by $\ECMV_n = \chi_{[-n,n]}\ECMV$. Then, we define $\dd k_n$ to be the normalized eigenvalue counting measure; that is, $\dd k_n$ puts a Dirac atom of weight $m/(2n+1)$ at $\zeta$ whenever $\zeta$ is an eigenvalue of $\ECMV_n$ having multiplicity $m$. Whenever $\dd k_n$ enjoys a weak$^*$ limit as $n \to \infty$, we refer to said limit as the \emph{density of states measure} (henceforth: DOS) of $\ECMV$ and denote it by $\dd k$. This is also closely related to the so-called \emph{density of zeros measure} for orthogonal polynomials on the unit circle (OPUC); cf.\ \cite[Proposition~8.2.1 and Theorem~10.5.21]{S1}.

\emph{Ergodic} CMV operators (sometimes called \emph{stochastic} CMV operators) supply an important class of examples for which the DOS exists. Concretely, let $S:\Omega \to \Omega$ be an invertible transformation of a Borel space $\Omega$. Given a measurable function $f:\Omega \to \D$, we may define CMV operators indexed by $\Omega$ via $\ECMV_\omega = \ECMV_{\alpha(\omega)}$, where
\[
\alpha_n(\omega)
=
f(S^n\omega),
\quad
n \in \Z, \; \omega \in \Omega.
\]
Then, if $\mu$ is an $S$-ergodic measure on $\Omega$ and $|f| \leq C < 1$ $\mu$-almost everywhere, then the DOS of $\ECMV_\omega$ exists for $\mu$-a.e.\ $\omega \in \Omega$ by ergodicity. Moreover, as demonstrated in \cite[Theorem~10.5.21]{S2}, the DOS of an element of such an ergodic family is $\mu$-almost surely given by the $\mu$-average of spectral measures:
\[
\int g \, \dd k
=
\int_\Omega \langle \delta_0, g(\ECMV_\omega) \delta_0 \rangle \, \dd \mu(\omega).
\]

We will use $\E(\cdot)$ to denote integration against $\mu$, that is,
\[
\E(f)
=
\int_\Omega f(\omega) \, \dd \mu(\omega)
\]
for $f \in L^1(\Omega,\dd \mu)$. We denote the $\delta_0$ spectral measure of $\ECMV_\omega$ by $\dd \nu_\omega$, i.e.,
\begin{equation} \label{eq:nuomegadef}
\langle \delta_0, g(\ECMV_\omega) \delta_0 \rangle
=
\int_{\X^1} g(z) \, \dd \nu_\omega(z),
\end{equation}
where $\X^1 = \partial \D$ denotes the unit circle. Finally, we define the \emph{Lyapunov exponent} by
\[
\gamma(z)
=
\lim_{n \to \infty} \frac{1}{n} \E(\log\| A_z^n\|),
\]
where $A_z^n$ denotes the Szeg\H{o} cocycle at spectral parameter $z \in \C$, that is,
\[
A_z^n(\omega)
=
\begin{bmatrix}
z & -\overline{\alpha_{n-1}(\omega)} \\
-\alpha_{n-1}(\omega) z & 1
\end{bmatrix}
\times
\cdots
\times
\begin{bmatrix}
z & -\overline{\alpha_{0}(\omega)} \\
-\alpha_0(\omega) z & 1
\end{bmatrix},
\quad
n \geq 1, \; \omega \in \Omega.
\]
A crucial role is played by the set on which $\gamma$ vanishes:
\[
\mathcal Z
\eqdef
\set{z \in \C : \gamma(z) = 0}.
\]

We let $\nu_\omega^{(\ac)}(z)$ denote the density of the absolutely continuous part of $\dd\nu_\omega$, and $k^{(\ac)}(z)$ the density of the absolutely continuous part of $\dd k$. Our main result is that $k^{(\ac)}$ is precisely the $\mu$-average of $\nu_\omega^{(\ac)}$ almost everywhere on $\mathcal Z$.

\begin{thm} \label{thm:kotaniavg}
For Lebesgue-almost every $z\in\mathcal Z$,
\begin{equation}
k^{(\ac)}(z)
=
\E\!\left(\nu_\omega^{(\ac)}(z)\right).
\end{equation}
\end{thm}

As a consequence, we deduce Kotani's characterization of (almost-sure) pure a.c.\ spectrum.

\begin{corollary} \label{coro:kotani}
The spectral type of $\ECMV_\omega$ is purely absolutely continuous for $\mu$-almost every $\omega\in\Omega$ if and only if the density of states measure is absolutely continuous and the Lyapunov exponent vanishes almost everywhere with respect to the density of states measure.
\end{corollary}

We note that the hypotheses of this corollary are satisfied in the case where the Verblunsky coefficients form a periodic sequence. We refer the reader to \cite[Chapter 11]{S2} for further details.

Establishing connections between the DOS measure and the spectral measure is important, because this allows us to extract facts about the spectral measure of a CMV operator based on an eigenvalue analysis of the finite sub-matrices of the CMV matrix. An example where this technique was useful is in \cite{Fillman-Ong}, where we were able to prove results regarding the spreading behavior of a limit-periodic quantum walk model by using the DOS measure of the CMV matrix to characterize its spectral measure.

\subsection*{Acknowledgements} We are thankful to David Damanik for helpful conversations. J.~F.\ was supported in part by an AMS--Simons Travel Grant 2016--2018.

\section{Proof of Main Theorem}
We will follow Kotani's original argument as presented in the proof  of \cite[Theorem~5]{Damanik.KotaniSurvey}. The broad strokes of the argument are similar; the main differences arise from the more complicated nature of Weyl--Titchmarsh theory for CMV matrices. Concretely, there are several different analytic functions that play the role of the Weyl--Titchmarsh $m$-function, which complicates some of the algebraic gymnastics. For more on the various analogs of the $m$-function for CMV operators, see \cite{simanalogs}.

We define first the Green's  function 
\begin{equation} \label{eq:greenfctdef}
G_\omega(z)
=
G_\omega(0,0;z)
\eqdef
\left< \delta_0, (\mathcal E_\omega-z)^{-1}\delta_0\right>,
\quad
z \in \C \setminus \sigma(\ECMV_\omega).
\end{equation}
In view of the definition \eqref{eq:nuomegadef}, one immediately has
\[
\int_{\X^1} \frac{\dd \nu_\omega(\tau)}{\tau-z}
=
G_\omega(z).
\]
The Carath\'eodory function of $\nu_\omega$ will be defined by
\begin{equation} \label{eq:carafctdef}
F_\omega(z)
=
\left< \delta_0, (\mathcal E_\omega+z)(\mathcal E_\omega-z)^{-1}\delta_0\right>,
\quad
z \in \C \setminus \sigma(\ECMV_\omega).
\end{equation}
This defines an analytic function from $\D$ to the right half plane, whose limiting behavior on the unit circle is connected to the behavior of $\dd \nu_\omega$. Critically, one can recover the absolutely continuous part of $\dd \nu_\omega$ from the boundary values of $\mathrm{Re}\, F_\omega$. That is, by \cite[(1.3.31)]{S2} we have
\begin{equation}\label{eq:nuacfromcara}
\nu_\omega^{(\ac)}(e^{i\theta})
=
\lim_{r\uparrow 1} \frac{1}{2\pi} \mathrm{Re}\, F_\omega(re^{i\theta})
\end{equation}
for Lebesgue almost every $\theta \in [0,2\pi)$.
 
Using \eqref{eq:nuomegadef}, \eqref{eq:greenfctdef}, and \eqref{eq:carafctdef} we may relate the Carath\'eodory and Green functions via
\begin{equation} \label{eq:caragreenrel}
F_\omega(z)
=
\int_{\X^1} \frac{\tau + z}{\tau - z} \, \dd \nu_\omega(\tau)
=
\int_{\X^1} \left(1 + \frac{2z}{\tau-z} \right) \, \dd \nu_\omega(\tau)
=
1 + 2z G_\omega(z).
\end{equation}

Let us define $\Gamma$ as in \cite[(10.11.18)]{S2}, i.e.,
\[
\Gamma(z)
=
\int_{\X^1} \log\left(\frac{1-z \overline{\tau}}{\rho_\infty}\right) \, \dd k(\tau),
\text{ where }
\rho_\infty
=
\exp\left(\frac{1}{2} \E \! \left(\log(1-|f(\omega)|^2) \right)\right).
\]
By the ergodic theorem, one has
\[
\lim_{n \to \infty} \left(\prod_{j=0}^{n-1} (1-\vert \alpha_j(\omega)\vert^2)\right)^{1/2n}
=
\rho_\infty
\text{ for } \mu\text{-almost every }\omega \in \Omega.
\]
The Lyapunov exponent satisfies the Thouless formula. That is, we have
\begin{equation} \label{eq:thouless}
\gamma(z) = \mathrm{Re} \, \Gamma(z)
\end{equation} for all $z$.\footnote{There is a small subtlety here: we are using $\gamma$ to denote the averaged Lyapunov exponent, which always exists and obeys the Thouless Formula for all $z$. The behavior of the non-averaged Lyapunov exponent can be somewhat delicate on $\X^1$; see the discussion in the Remark following \cite[Theorem~10.5.26]{S2} for additional details.}

We define also the Carath\'eodory function corresponding to the DOS by
\[
K(z)
=
\int_{\X^1} \frac{\tau+z}{\tau-z} \, \dd k(\tau),
\quad
z \in \D.
\]
Just as with $F_\omega$, we may use the boundary values of $K$ to recover the absolutely continuous part of $\dd k$ as in \eqref{eq:nuacfromcara}. One has 
\begin{equation}\label{eq:kacfromcara}
k^{(\ac)}(e^{i\theta})
=
\frac{1}{2\pi}\lim_{r\uparrow 1} \mathrm{Re}\, K(re^{i\theta})
\end{equation}
for a.e.\ $\theta \in [0,2\pi)$. By the definitions of the functions $K$ and $\Gamma$, it is straightforward to calculate that they are connected via
\begin{equation}\label{e:gammaK}
K(z)
=
1-2z\dfrac{\dd \Gamma}{\dd z}(z),
\end{equation} 
which may be viewed as an analogue of \eqref{eq:caragreenrel}.

\begin{proof}[Proof of Theorem~\ref{thm:kotaniavg}]
The ``$\geq$" direction follows immediately from \cite[Theorem~10.11.11]{S2}, and the fact that the average of absolutely continuous measures is absolutely continuous.

For the other direction, we use (10.11.28) of \cite{S2}, which states that
\begin{equation}\label{e:averageG}
\E(G_\omega(z))
=
\int_{\X^1} \frac{\dd k(\tau)}{\tau-z}.
\end{equation}

For convenience, let us introduce $\mathcal Z^\circ \subseteq [0,2\pi)$ by insisting that $\theta \in \mathcal Z^\circ$ if and only if $e^{i\theta} \in \mathcal Z$. By \eqref{eq:thouless}, \eqref{e:gammaK}, and standard facts about Carath\'eodory functions (e.g.\ \cite[Section~1.3]{S1}), one has
\begin{equation}\label{e:limitingLyapunov}
\lim_{r\uparrow 1}\frac{\partial \gamma}{\partial r}(re^{i\theta})
=
\lim_{r\uparrow 1} \frac{\gamma(e^{i\theta}) - \gamma(re^{i\theta})}{1-r}
=
-\lim_{r\uparrow 1}\frac{\gamma(re^{i\theta})}{1-r}
\end{equation}
for Lebesgue almost every $\theta \in \mathcal Z^\circ$. Since it is critical to our proof, let us point out that there is a sign error in \cite[(10.11.22)]{S2}, which is why our Equation~\eqref{e:limitingLyapunov} does not match Simon's identity for the boundary values of $\partial \gamma/\partial r$ on $\mathcal Z$.

We can then write
\begin{align}
k^{(\ac)}(e^{i \theta})
& =
\lim_{r\uparrow 1} \frac{1}{2\pi}\mathrm{Re} \, K(re^{i \theta}) \text{ by \eqref{eq:kacfromcara} }\nonumber\\
& =
\lim_{r\uparrow 1} \frac{1}{2\pi}\mathrm{Re}\left( 1-2re^{i \theta}\dfrac{\dd \Gamma}{\dd z}(re^{i \theta})\right) \text{ by \eqref{e:gammaK} }\nonumber\\
& =
\frac{1}{2\pi}-
\lim_{r\uparrow 1} \frac{1}{\pi} \frac{\partial \gamma}{\partial r}(r e^{i \theta})
 \text{ by \eqref{eq:thouless} and Cauchy--Riemann} \nonumber\\
& =
\frac{1}{2\pi}+
\lim_{r\uparrow 1}\frac{1}{\pi}\left(\frac{\gamma(re^{i \theta})}{1-r} \right) \text{ by \eqref{e:limitingLyapunov}}
\label{e:bigcalc}
\end{align}
for Lebesgue a.e.\ $\theta \in \mathcal Z^\circ$.

Let $f_+(z) =f_+(z,\omega)$ and $f_-(z) =f_-(z,\omega)$ be the Schur functions corresponding to the half-line CMV matrices with Verblunsky coefficient sequences given by $\alpha_0(\omega), \alpha_1(\omega),\ldots$ and $-\overline{\alpha_{-1}(\omega)}, -\overline{\alpha_{-2}(\omega)}, \ldots$ respectively. A bit more precisely, these coefficient sequences determine half-line operators $\CMV_\pm(\omega)$ as in \eqref{eq:halflinecmvdef} with cyclic vector $\delta_0$ and associated spectral measures $\dd\mu_{\pm,\omega}$. The Schur functions are analytic functions from $\D$ to itself defined in terms of the half-line Carath\'eodory functions $K_{\pm,\omega}$ by:
\[
f_\pm(z,\omega)
=
\frac{1}{z} \frac{K_{\pm,\omega}(z)-1}{K_{\pm,\omega}(z)+1},
\quad
K_{\pm,\omega}(z)
=
\int_{\X^1} \frac{\tau+z}{\tau-z} \, \dd\mu_{\pm,\omega}(\tau).
\] 
For further discussion on the role of the Schur function in the spectral theory of CMV matrices, we refer the interested reader to \cite[Section~1.3]{S1}.

These half-line Schur functions are connected to the Lyapunov exponent via \cite[Proposition 10.11.7]{S2}, which gives us:
\begin{equation}\label{e:gamma->f}
\gamma(z)
=
\frac{1}{2}\E\left( \log\left( \frac{1-\vert zf_+\vert^2}{1-\vert f_+ \vert^2}  \right)\right).
\end{equation}

Additionally, by \cite[Proposition 10.11.12]{S2}, the Schur functions are connected to the Green's function in the following way:
\begin{equation}\label{e:Green->Schur}
G_{\omega}(z)
=
\frac{f_+(z,\omega)f_-(z,\omega)}{1-zf_+(z,\omega)f_-(z,\omega)}.
\end{equation}
In view of \eqref{eq:caragreenrel}, this implies
\begin{equation}\label{e:Cara->Schur}
F_{\omega}(z)
=
\frac{1+zf_+(z,\omega)f_-(z,\omega)}{1-zf_+(z,\omega)f_-(z,\omega)}.
\end{equation}

Moreover, by \cite[Theorem 10.11.16]{S2}, $\ECMV_\omega$ is reflectionless on $\mathcal Z$ for $\mu$ a.e.\ $\omega \in \Omega$. That is, for $\mu$ a.e.\ $\omega$, we have
\begin{equation}\label{e:f_+=f_-}
f_+(z_0,\omega)
=
\overline{z_0f_-(z_0,\omega)}
\end{equation}
for Lebesgue a.e.\ $z_0\in \mathcal Z$.  Consequently, for $\mu$ a.e.\ $\omega$, the following calculation holds for Lebesgue a.e.\ $\theta \in \mathcal Z^\circ$:
\begin{align}
\nu_\omega^{(\ac)}(e^{i \theta})
& =
\lim_{r\uparrow 1}\frac{1}{2\pi}\mathrm{Re} \, F_\omega(re^{i \theta}) \text{ by \eqref{eq:nuacfromcara} } \nonumber\\
& =
\lim_{r\uparrow 1}\frac{1}{2\pi}\mathrm{Re}\left(
\frac{1+re^{i \theta}f_+(re^{i \theta})f_-(re^{i \theta})}{1-re^{i \theta} f_+(re^{i \theta})f_-(re^{i \theta})}\right)\text{ by \eqref{e:Cara->Schur}}\nonumber\\
& =
\lim_{r\uparrow 1}\frac{1}{2\pi}\left(
\frac{1+\vert f_+(re^{i \theta})\vert^2}{1-\vert f_+(re^{i \theta})\vert^2}\right)\text{ by \eqref{e:f_+=f_-}}\label{e:nu_omega^ac}.
\end{align}

Let $P_R$ be the Poisson kernel  for the unit disk, that is, for $R\in[0,1)$ and $\tau \in \X^1$, put
\[
P_R(\tau)
=
\mathrm{Re}\left(
\frac{1+R\tau}{1-R\tau}
\right).
\]
Throughout the rest of the argument, we will freely use some basic facts about Poisson integrals and boundary values of harmonic functions. The reader is invited to consult \cite[Chapter~11]{Rudin} for further information. Let us also define
\[
C_R(\phi)
=
\int _{\mathcal Z^\circ} P_R(e^{i(\phi-\phi')}) \, \dd\phi',
\]
and
\[
\widetilde P_R(\phi,\phi')
=
P_R(e^{i(\phi-\phi')})C_R(\phi)^{-1}.
\]

By Jensen's inequality, we have 
\begin{align*}
\int_{0}^{2\pi}\nu_\omega^{(\ac)}(e^{i\phi'})P_R(e^{i(\phi-\phi')}) \,\dd\phi'
& \geq 
\int_{\mathcal Z^\circ}\nu_\omega^{(\ac)}(e^{i\phi'})P_R(e^{i(\phi-\phi')}) \,\dd\phi' \\
& =
C_R(\phi)\int_{\mathcal Z^\circ} \nu_\omega^{(\ac)}(e^{i\phi'}) \widetilde P_R(\phi,\phi') \,\dd\phi' \\ 
& \geq
C_R(\phi) \left(\int_{\mathcal Z^\circ} \left(\nu_\omega^{(\ac)}(e^{i\phi'}) \right)^{-1} \widetilde P_R(\phi,\phi') \,\dd\phi'\right)^{-1} \\ 
& \geq
C_R(\phi)^2\left(\int_{0}^{2\pi} \left( \nu_\omega^{(\ac)}(e^{i\phi'})\right)^{-1}  P_R(e^{i(\phi-\phi')}) \,\dd\phi'\right)^{-1}
\end{align*}
for every $\omega$ and every $\phi$. Then, by \eqref{e:nu_omega^ac}, we get
\[
\int_{0}^{2\pi}\nu_\omega^{(\ac)}(e^{i\phi'})P_R(e^{i(\phi-\phi')}) \,\dd\phi'
\geq
\frac{C_R(\phi)^2}{2\pi} \left(
\frac{1+\vert f_+(Re^{i\phi})\vert^2}{1-\vert f_+(Re^{i\phi})\vert^2}\right)
\]
for $\mu$ a.e.\ $\omega$ and Lebesgue a.e.\ $\phi \in \mathcal Z^\circ$. Consequently, for a.e.\ $\phi \in \mathcal Z^\circ$, we get
\[
\int_{0}^{2\pi} \E\left(\nu_\omega^{(\ac)}(e^{i\phi'})\right)P_R(\phi-\phi') \,\dd \phi'
\geq 
\frac{C_R(\phi)^2}{2 \pi} 
\E \left(
\frac{1+\vert f_+(Re^{i\phi'})\vert^2}{1-\vert f_+(Re^{i\phi'})\vert^2}\right).
\]
Sending $R \uparrow 1$, we note that $|C_R(\phi)| \leq 1$ for all $\phi$ and $C_R(\phi)\to 1$ for Lebesgue a.e.\ $\phi \in \mathcal Z^\circ$ by standard properties of Poisson integrals. Consequently, after re-labeling $R$ as $r$, we get
\begin{equation}\label{e:exp(nu_omega^ac)}
\E\!\left(\nu_\omega ^{(\ac)}(e^{i\phi}) \right)
\geq
\frac{1}{2\pi}\limsup_{r\uparrow 1} 
\E\left( \frac{1+\vert f_+(re^{i\phi})\vert^2}{1-\vert f_+(re^{i\phi}) \vert^2} \right)
\end{equation}
for Lebesgue a.e.\ $\phi \in \mathcal Z^\circ$.

At last, we put everything together. By \eqref{e:bigcalc} and \eqref{e:gamma->f}, we have
\[
k^{(\ac)}(e^{i\phi})
=
\frac{1}{2\pi} +
\lim_{r\uparrow 1}\frac{1}{2\pi(1-r)} \E\left(\log\left( 1+\frac{\vert f_+(re^{i\phi})\vert^2 (1-r^2)}{1-\vert f_+(re^{i\phi})\vert^2}\right)\right)
\] 
for a.e.\ $\phi \in \mathcal Z^\circ$. Since $\log(1+x) \leq x$ for $x \geq 0$, we get
\begin{align*}
k^{(\ac)}(e^{i\phi})
& \leq
\frac{1}{2\pi} + \frac{1}{2\pi}
\limsup_{r\uparrow 1}\E 
\left( \frac{\vert f_+(re^{i\phi})\vert^2(1+r)}{1-\vert f_+(re^{i\phi})\vert^2} \right) \\
& =
\frac{1}{2\pi} + \frac{1}{\pi} \limsup_{r\uparrow 1}
\E \left( \frac{\vert f_+(re^{i\phi})\vert^2}{1-\vert f_+(re^{i\phi})\vert^2}\right)  \\
& =
\frac{1}{2\pi}
\limsup_{r\uparrow 1} \E \left( \frac{1+\vert f_+(re^{i\phi})\vert^2}{1-\vert f_+(re^{i\phi})\vert^2}\right)
\end{align*}
for a.e.\ $\phi \in \mathcal Z^\circ$. Combining this with \eqref{e:exp(nu_omega^ac)} we get
\[
k^{(\ac)}(e^{i\phi})
\leq
\frac{1}{2\pi}\limsup_{r\uparrow 1}\E \left( \frac{1+\vert f_+(re^{i\phi})\vert^2}{1-\vert f_+(re^{i\phi})\vert^2}\right)
\leq
\E\!\left(\nu_\omega ^{(\ac)}(e^{i\phi}) \right)
\]
for a.e.\ $\phi \in \mathcal Z^\circ$, which completes the proof.

\end{proof}

\begin{proof}[Proof of Corollary~\ref{coro:kotani}]
Suppose first that the spectrum of $\ECMV_\omega$ is purely absolutely continuous for $\mu$-almost every $\omega \in \Omega$. Then, as it is the average of absolutely continuous measures, it follows that $\dd k$ is absolutely continuous (cf.\ \cite[(10.5.50)--(10.5.51)]{S2}). Moreover, since the almost-sure absolutely continuous spectrum is given by the essential closure of the set upon which the Lyapunov exponent vanishes (e.g.\ by \cite[Theorem~10.11.1]{S2}), it follows that the Lyapunov exponent vanishes Lebesgue-a.e.\ (hence $\dd k$-a.e.) on the spectrum.

Conversely, if $\dd k$ is purely a.c.\ and $\gamma$ vanishes $\dd k$-a.e. on $\X^1$, then
\begin{align*}
1 
& =
\int_\Sigma \dd k(z) \\
& =
\int_{\mathcal Z} \dd k(z) \\
& =
\int_{\mathcal Z} k^{(\ac)}(z) \, \dd \lambda(z) \\
& =
\int_{\mathcal Z} \E(\nu_\omega^{(\ac)}(z)) \, \dd \lambda(z) \\
& =
\E \left( \int_{\mathcal Z} \nu_\omega^{(\ac)}(z) \, \dd \lambda(z) \right) \\
& \leq
\E \left( \int_{\Sigma} \nu_\omega^{(\ac)}(z) \, \dd \lambda(z) \right) \\
& \leq 
1,
\end{align*}
by Theorem~\ref{thm:kotaniavg} and Fubini's Theorem. In the calculation above, $\dd\lambda$ denotes normalized 1D Lebesgue measure on $\X^1$. Since the chain of inequalities begins and ends with one, all inequalities are equalities, so the absolutely continuous part of $\nu_\omega$ has full weight for $\mu$-a.e.\ $\omega$. The transformation $S$ preserves $\mu$, so one also has $\nu_{S\omega}^{(\ac)}(\Sigma)= 1$ for $\mu$-a.e.\ $\omega$. Since $\nu_{S\omega}$ is the spectral measure of $\ECMV_\omega$ corresponding to the vector $\delta_1$ and the pair $\{\delta_0,\delta_1\}$ is cyclic for $\ECMV_\omega$, taking the intersection of those two full-measure sets yields a full-measure set for which $\ECMV_\omega$ has purely a.c.\ spectrum.
\end{proof}

\end{document}